\documentclass{siamonline171218}
\usepackage{amssymb,amsfonts,amsmath}

\usepackage{mathabx}

\usepackage[mathscr]{eucal}

\usepackage{amsbsy}

\usepackage{stmaryrd}

\usepackage{bm}

\usepackage{xparse}

\usepackage{subcaption}

\usepackage{braket}

\usepackage{pgfplots}
\pgfplotsset{compat=1.14}
\usepackage{siunitx}

\NewDocumentCommand{\LineFor}{m m}{%
  \State\textbf{for} {#1}, \textbf{do} {#2}, \textbf{end}
}

\newcommand{\mch}[1]{\multicolumn{1}{c|}{{#1}}}
\newcommand{\mchm}[1]{\multicolumn{1}{c|}{{$#1$}}}

\DeclareMathOperator{\diag}{diag}

\DeclareMathOperator{\rank}{rank}

\NewDocumentCommand \qtext {m} {\quad\text{#1}\quad}

\NewDocumentCommand \Real {} {\mathbb{R}}

\NewDocumentCommand \Rdn {} {\mathbb{R}^{[d,n]}}

\NewDocumentCommand \Sdn {} {\mathbb{S}^{[d,n]}}

\NewDocumentCommand \Kdnr {} {\mathbb{K}^{[d,n,r]}}
\NewDocumentCommand \Kdnp {} {\mathbb{K}^{[d,n,p]}}
\NewDocumentCommand \Kdns {} {\mathbb{K}^{[d,n,s]}}
\NewDocumentCommand \Dnvec {} {\mathbf{D}_{\nvec}}
\NewDocumentCommand \Dlvec {} {\mathbf{D}_{\lvec}}

\NewDocumentCommand \T { O{} m } {\boldsymbol{#1\mathscr{\MakeUppercase{#2}}}}

\NewDocumentCommand \Mx { O{} m } {{\bm{#1\mathbf{\MakeUppercase{#2}}}}} 

\NewDocumentCommand \Vc { O{} m } {{\bm{#1\mathbf{\MakeLowercase{#2}}}}} 

\NewDocumentCommand \dotprod {m m}{\big< #1 , #2 \big>} 

\NewDocumentCommand \X { } {\T{X}}
\NewDocumentCommand \Xs { } {\T[\tilde]{X}}

\NewDocumentCommand \xe { s } 
{
  \IfBooleanTF{#1}
  {x_{i_1 i_2 \dots i_d}}
  {x_{i}}
}

\NewDocumentCommand \bigsum { G{1} }
{
  \sum_{i_{#1}=1}^n \cdots \sum_{i_d=1}^n
}

\NewDocumentCommand \psum { } {\sum_{\ell=1}^p}

\NewDocumentCommand \rsum { } {\sum_{j=1}^r}

\NewDocumentCommand \V {} {\Mx{V}}
\NewDocumentCommand \Vs {} {\Mx[\tilde]{V}}

\NewDocumentCommand \Vl { G{\ell} } {\Vc{v}_{#1}}

\NewDocumentCommand \Vbar { } {\Vc[\bar]{v}}

\NewDocumentCommand \VPI { O{\Vbar} G{i} }{%
  #1 \otimes \Vc{e}_{#2} \otimes \Vc{e}_{#2}    
  + \Vc{e}_{#2} \otimes #1 \otimes \Vc{e}_{#2}    
  + \Vc{e}_{#2} \otimes \Vc{e}_{#2} \otimes #1
}

\NewDocumentCommand \Vlod { G{\ell} O{d} } {\Vc{v}_{#1}^{\otimes #2}}

\NewDocumentCommand \Vikl {O{k} G{\ell} } {v_{i_{#1} #2}}

\NewDocumentCommand \Vil { G{\ell} } {v_{i #1}}

\NewDocumentCommand \nvec {} {\Vc{\nu}}

\NewDocumentCommand \nl { G{\ell} } {\nu_{#1}}

\NewDocumentCommand \xterm { } {\nl^{\phantom{\otimes d}} \!\!\!  \Vlod}

\NewDocumentCommand \M {} {\T[\hat]{X}}

\NewDocumentCommand \Mk { O{k} } {\Mx[\hat]{X}_{(#1)}}

\DeclareDocumentCommand \me { s } 
{
  \IfBooleanTF{#1}
  {\hat x_{i_1 i_2\dots i_d}}
  {\hat x_{i}}
}

\NewDocumentCommand \A {} {\Mx{A}}

\NewDocumentCommand \Aj { G{j} } {\Vc{a}_{#1}}
\NewDocumentCommand \Ajod { G{j} } {\Vc{a}_{#1}^{\otimes d}}

\NewDocumentCommand \Aji { G{j} } {\Vc{\mu}_{#1}} %

\NewDocumentCommand \Aikj {O{k} G{j} } {a_{i_{#1} #2}}

\NewDocumentCommand \Y {} {\Mx{Y}}
\NewDocumentCommand \Yj {} {\Vc{y}_j}

\NewDocumentCommand \lvec {} {\Vc{\lambda}}

\NewDocumentCommand \lj { G{j} } {\lambda_{#1}}

\NewDocumentCommand \mterm { } {\lj^{\phantom{\otimes d}} \!\!\! \Ajod}

\NewDocumentCommand \gl {} {\Vc{g}_{\Vc{\lambda}}}
\NewDocumentCommand \gA {} {\Mx{G}_{\A}}

\NewDocumentCommand \Ak { G{k} t' t"  } { \Mx{A}_{#1}\IfBooleanTF{#2}{^{\intercal}}{}\IfBooleanTF{#3}{^{\phantom{\intercal}}}{} }

\NewDocumentCommand \AkAkt { G{k} } {\Ak{#1}'\Ak{#1}"}

\NewDocumentCommand \Ake { G{k} G{i} G{j} } {
  a_{#1}(#2_{#1},#3)
}

\NewDocumentCommand \KT { s } {
  \llbracket 
  \IfBooleanTF{#1}{\lvec;}{}
  \Ak{1}, \Ak{2}, \dots,  \Ak{d} \rrbracket
}

\NewDocumentCommand \RelErrPlot {m m} {%
  \begin{tikzpicture}
    \pgfmathsetmacro\presult{#2*250}
    \begin{axis}[
      legend columns=-1,
      legend entries={$\sigma=10^{-4}$,$\sigma=10^{-3}$,$\sigma=10^{-2}$, $\sigma=10^{-1}$, true $r$},
      legend to name=relegend#1#2,
      title = {$r=#2,p=\pgfmathprintnumber{\presult}$},
      ylabel = rel. error,
      ymode = log,
      ytick distance = 10,
      ymax = 0.9,
      ymin = 9e-9,
      ]
      \addplot table[x=r#2_testvals,y=d#1_r#2_s10e-4_relerr] {\gmmtable};
      \addplot table[x=r#2_testvals,y=d#1_r#2_s10e-3_relerr] {\gmmtable};
      \addplot table[x=r#2_testvals,y=d#1_r#2_s10e-2_relerr] {\gmmtable};
      \addplot table[x=r#2_testvals,y=d#1_r#2_s10e-1_relerr] {\gmmtable};
      \addplot +[black,dashed,thick,no marks] coordinates {(#2,1e-9) (#2,0.9)};
    \end{axis}
  \end{tikzpicture}%
}

\NewDocumentCommand \ScorePlot {m m} {%
  \begin{tikzpicture}
    \pgfmathsetmacro\presult{#2*250}
    \begin{axis}[
      ylabel = sim.~score,
      ytick distance = 0.1,
      ymax = 1,
      ymin = 0.8,
      ]
      \addplot table[x=r#2_testvals,y=d#1_r#2_s10e-4_score] {\gmmtable};
      \addplot table[x=r#2_testvals,y=d#1_r#2_s10e-3_score] {\gmmtable};
      \addplot table[x=r#2_testvals,y=d#1_r#2_s10e-2_score] {\gmmtable};
      \addplot table[x=r#2_testvals,y=d#1_r#2_s10e-1_score] {\gmmtable};
      \addplot +[black,dashed,thick,no marks] coordinates {(#2,0) (#2,1)};
    \end{axis}
  \end{tikzpicture}%
}

\NewDocumentCommand \TimePlot {m m m} {%
  \begin{tikzpicture}
    \pgfmathsetmacro\presult{#2*250}
    \begin{axis}[
      xlabel = approx.~rank ($\hat r$),
      ylabel = time,
      xtick distance = 1,
      ymin = 0,
      ymax = #3,
      ]
      \addplot table[x=r#2_testvals,y=d#1_r#2_s10e-4_time] {\gmmtable};
      \addplot table[x=r#2_testvals,y=d#1_r#2_s10e-3_time] {\gmmtable};
      \addplot table[x=r#2_testvals,y=d#1_r#2_s10e-2_time] {\gmmtable};
      \addplot table[x=r#2_testvals,y=d#1_r#2_s10e-1_time] {\gmmtable};
      \addplot +[black,dashed,thick,no marks] coordinates {(#2,0) (#2,#3)};
    \end{axis}
  \end{tikzpicture}%
}

\NewDocumentCommand \RecoveryPlot {m} {%
  \begin{tikzpicture}
    \begin{axis}[xmin=-3, xmax=1.5, ymin=-2.5, ymax=0.5,
      title={$\hat r=#1$},
      legend columns=-1,
      legend entries={~data~~~,,,,,,~mean~~~,,,,,~recovered mean},
      legend to name=mlegend#1]
      \addplot[only marks, mark size=0.4pt] table[meta=meta] {coords.dat};
      \addplot[only marks, mark=x, mark size=5pt, line width=3pt] table[meta=meta] {truemeans.dat};
      \addplot[only marks] table[meta=meta] {means#1.dat};
    \end{axis}
  \end{tikzpicture}%
}

\usepackage{cleveref}
\usepackage{algorithm}
\usepackage{algpseudocode}
\usepackage{multirow}

\title{Estimating Higher-Order Moments Using Symmetric Tensor
  Decomposition%
  \thanks{%
    This material is based upon work for which the first author was
    supported by the National Science Foundation Mathematical Sciences Graduate
    Internship (MSGI) Program and the second author was
    supported by the U.S. Department of Energy, Office of Science,
    Office of Advanced Scientific Computing Research (ASCR) Applied
    Mathematics Program.
    Sandia National Laboratories is a multimission laboratory
    managed and operated by National Technology and Engineering
    Solutions of Sandia, LLC., a wholly owned subsidiary of Honeywell
    International, Inc., for the U.S.\@ Department of Energy's National
    Nuclear Security Administration under contract DE-NA-0003525.
    This paper describes objective technical results and analysis. Any
    subjective views or opinions that might be expressed in the paper
    do not necessarily represent the views of the U.S.\@ Department of
    Energy or the United States Government.}}
\author{%
  Samantha Sherman%
  \thanks{University of Notre Dame, Notre Dame, IN (\email{ssherma1@nd.edu})}
  \and
  Tamara G. Kolda%
  \thanks{Sandia National Laboratories, Livermore, CA (\email{tgkolda@sandia.gov})}
  }

\headers{Higher-Order Moments Using Symmetric Tensor Decomposition}{Samantha Sherman and Tamara G. Kolda}

\begin{document}
\maketitle

\begin{abstract}
  We consider the problem of decomposing higher-order moment tensors,
  i.e., the sum of symmetric outer products of data vectors.
  Such a decomposition can be used to estimate the
  means in a Gaussian mixture model and for other applications
  in machine learning.
  The $d$th-order empirical moment tensor of a set of $p$ observations of $n$ variables
  is a symmetric $d$-way tensor.
  Our goal is to find a low-rank tensor approximation
  comprising $r \ll p$ symmetric outer products.
  The challenge is that forming the empirical moment tensor costs  $O(pn^d)$ operations and $O(n^d)$ storage,
  which may be prohibitively expensive; additionally, the algorithm to
  compute the low-rank approximation costs  $O(n^d)$  per iteration.
  Our contribution is  avoiding formation of the
  moment tensor, computing the low-rank tensor approximation of the moment
  tensor \emph{implicitly} using $O(pnr)$ operations per iteration and no extra memory.
  This advance opens the door to more applications of higher-order
  moments since they can now be efficiently computed.
  We present numerical evidence of the computational savings and show
  an example of estimating the means for higher-order moments. 
\end{abstract}


\section{Introduction}\label{sec:Intro}

Moments and cumulants are commonly used in statistical analysis of random variables.
They are involved in testing normality of data, estimating parameters of distributions, detecting outliers, etc.~\cite{G.C.Casella2001,KoSr05,Ko08a,Mc18}.
Let $V \in \Real^n$ be a multivariate random variable.
The first moment and cumulant are simply the expected value, $\mathbb{E}(V)$.
The second moment is the expected value of the outer product of the random variable with itself, $\mathbb{E}(V \otimes V) \in \Real^{n \times n}$.
The second cumulant, also known as the covariance, is the expected value of the outer product of the \emph{centered} observations, i.e., with the mean subtracted off \cite{KoSr05,Mc18, McCullagh2009}.
In general, 
the $d$th moment is the expected value of the \emph{$d$-way} outer product of the random variable with itself, $\mathbb{E}(V \otimes V \otimes \cdots \otimes V)$, forming a $d$-way symmetric tensor of dimension $n$.
The $d$th cumulant is based on the $d$th moment and subtracts off appropriate lower-order effects, but we omit the formulas here and instead refer readers to McCullaugh \cite{Mc18, McCullagh2009}.
The third cumulant is often called the \emph{skewness} and measures the asymmetry of the distribution; for instance,
Gaussian distributions have zero skewness.
The fourth cumulant is called the kurtosis of a distribution and measures the sharpness or flatness  of the distribution~\cite{KoSr05}.

Tensor decomposition of moment and cumulant tensors are used in a variety of statistical and data science applications, including
independent component analysis and blind source separation \cite{Ca91a,Co02,DeDeVa01},
clustering \cite{ShZaHa06,CiJaZdAm07},
learning Gaussian mixture models \cite{HsKa13,AnBeGoRa14,GoVeXi14,GeHuKa15,SeJaAn16},
latent variable models \cite{AnGeHsKa14,AnGeHsKa14a},
outlier detection \cite{Do18,AdKoKeSh19},
feature extraction in hyperspectral imagery \cite{GeWa19},
and multireference alignment \cite{PeWeBaRi19}.
In these cases, it is assumed  that the empirical higher-order moment is already computed.
However, computing these moments is extremely expensive when $n$ is even moderately large and especially for $d \geq 4$ because
the cost is $O(pn^d)$ for $p$ observations.
There is some concern about this expense, e.g., Domino et al.~\cite{DoGaPa18} show how to reduce the complexity by a factor of $d!$ by
exploiting symmetry.
Our goal is to show how to avoid the formation expense altogether.

We let $\X$ denote the empirical estimate of the $d$th moment based on $p$ observations, of size $n^d$.
We consider the problem of computing a low-rank  symmetric canonical polyadic (CP) tensor decomposition, $\M$:
\begin{equation}\label{eq:cp_element}
  \X \approx \M \equiv \rsum \mterm
  \quad\Leftrightarrow\quad
 \me* = \rsum \lj \prod_{k=1}^d \Aikj \text{ for all } i_1,i_2,\dots,i_d \in \set{1,\dots,n},
\end{equation}
where $\me*$ is the $(i_1,i_2,\dots,i_d)$ entry of $\M$,
$r \ll p$ is the \emph{rank} of the approximation $\M$,
$\lj$ is the $j$th entry of the \emph{weight vector} $\lvec \in \Real^r$,
and $\Aj$ and $\Aikj$ are the $j$th column and $(i_k,j)$ entry, respectively, of the \emph{factor matrix} $\A \in \Real^{n \times r}$.
The superscript $\otimes d$ denotes the $d$-way outer product.
An illustration of the symmetric CP decomposition is shown in \cref{fig:symcp}.
The symmetric CP problem is well studied \cite{CoGoLiMo08,KyEr08,BeGiId11,Ni14,Ko15}.

\begin{figure}
  \centering
  \includegraphics[scale=0.6]{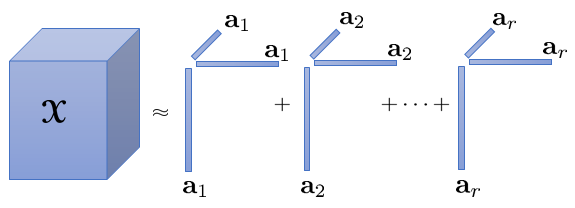}
  \caption{Illustration of the rank-$r$ low-rank approximation for a 3-way tensor ($d=3$).}
  \label{fig:symcp}
\end{figure}

We revisit the symmetric CP problem to consider the case
that the data tensor $\X$ has special structure as follows.
Suppose we are given a set of $p$ observations, denoted $\set{\Vl{1}, \Vl{2}, \dots, \Vl{p}}$,
of an $n$-dimensional real-valued random variable $V$. Then the $d$th empirical moment is given by
\begin{equation}\label{eq:moment_tensor_element}
  \X = \frac{1}{p} \psum \Vlod
  \quad \Leftrightarrow \quad
  \xe*  = \frac{1}{p} \psum \prod_{k=1}^d \Vikl \qtext{for all} i_1,i_2,\dots,i_d \in \set{1,\dots,n},
\end{equation}
where $\xe*$ is the $(i_1,i_2,\dots,i_d)$ entry of the moment tensor $\X$ and
$\Vikl$ denotes the $i_k$ entry of observation $\Vl$.
The tensor $\X$ is symmetric because it is invariant under any permutation of the indices, which
is a consequence of the order of multiplication being irrelevant.

The difficulty that we seek to overcome is that the storage and computational costs of forming and working with the moment tensor can be prohibitive.
The work to construct $\X$ is $O(pn^d)$, the storage is $O(n^d)$,
and the work per iteration to compute a rank-$r$ approximation is $O(rn^d)$.
For $d=3$ and $n=1000$, $\X$ requires 8~GB of storage;
for $d=4$ and $n=200$, $\X$ requires 12~GB of storage.%
\footnote{The storage cost in gigabytes is calculated as $n^d \cdot 8/10^9$. Technically, we could reduce the storage costs by a factor of $d!$ by exploiting symmetry \cite{BaKoPl11}, but this drives up the computational cost due to poor data locality.}
The per iteration storage and floating point operations per iteration cost to compute a rank $r=10$ approximation
are shown in \cref{fig:flops} for different values of $d$ and $n$.

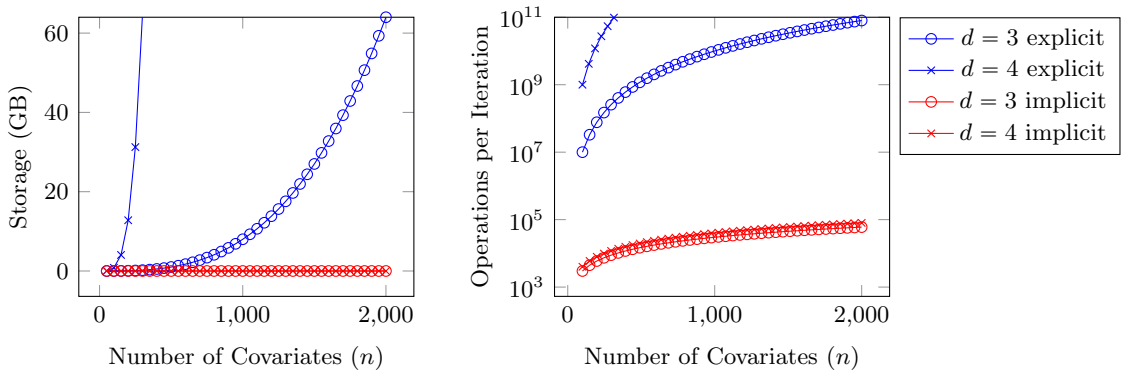
\begin{figure}
  \centering
    \begin{tikzpicture}
      \begin{axis}[
        ylabel shift = -3 pt,
        font=\footnotesize,
        scale = 0.65, transform shape,
        xlabel={Number of Covariates ($n$)},
        ylabel={Storage (GB)},
        cycle list name = color,
        ymax = 64,
        ]
        \addplot+[samples at={50,100,...,2000},mark=o,blue]{8*x^3/10^9};
        \addplot+[samples at={50,100,...,600},mark=x,blue]{8*x^4/10^9};
        \addplot+[domain=50:2000,samples=40,mark=o,red]{8*x*3*10/10^9};
        \addplot+[domain=50:2000,samples=40,mark=x,red]{8*x*4*10/10^9};
      \end{axis}
    \end{tikzpicture}
  ~~
    \begin{tikzpicture}
      \begin{semilogyaxis}[
        ylabel shift = -5 pt,
        font=\footnotesize,
        scale = 0.65, transform shape,
        xlabel={Number of Covariates ($n$)},
        ylabel={Operations per Iteration},
        legend entries = {$d=3$ explicit, $d=4$ explicit, $d=3$ implicit, $d=4$ implicit},
        legend pos = outer north east,
        cycle list name = color,
        ymax = 1e11,
        ]
        \addplot+[domain=100:2000,samples=40,mark=o,blue]{10*x^3};
        \addplot+[domain=100:400,samples=8,mark=x,blue]{10*x^4};
        \addplot+[domain=100:2000,samples=40,mark=o,red]{10*x*3};
        \addplot+[domain=100:2000,samples=40,mark=x,red]{10*x*4};
      \end{semilogyaxis}
    \end{tikzpicture}
    \caption{Comparison of storage required and floating-point operations per iteration
    to compute a rank $r=10$ symmetric canonical tensor decomposition
    approximation for $d$th moment for different numbers of covariates
    ($n$) when the empirical moment tensor is stored explicitly or
    implicitly.}
  \label{fig:flops}
\end{figure}

To address this problem, our contribution is to show that we can construct a low-rank approximation of a symmetric moment tensor without ever \emph{explicitly} forming it, which we refer to as \emph{implicit} computation of the symmetric CP decomposition.
This approach reduces the computational cost per iteration to $O(pnr)$. We can use the low-rank estimates of the moments to also produce approximations to the cumulants and do other calculations. Implicit calculation means that we can analyze much larger data sets.  
Avoiding formation of an expensive operator or array is a common practice in numerical methods, e.g., in ``matrix-free'' methods and other contexts.
For instance, a similar implicit approach is used in the context of (non-symmetric) tensor decomposition for computational chemistry
in which case the original tensor also has a sum-of-rank-one components format \cite{BeAuEsHa11}, but that work was specific to
reducing the cost of tensor contractions in the context of the computational chemistry application (post-Hartree-Fock electronic structure methods). 
To the best of our knowledge, the proposal to do implicit computation in the approximation of symmetric moment tensors is novel.

The paper is organized as follows. Notation and background on the symmetric tensor decomposition are established in~\cref{sec:Prelims}. The equivalence of the implicit and explicit approaches is shown in \cref{sec:ImpDecomp}.
In \cref{sec:Results}, we numerically demonstrate the computational savings and illustrate that larger problems are accessible, showing
examples based on estimating the means in a spherical Gaussian mixture model.
We discuss a stochastic approach that can be used for extremely large or online computation in \cref{sec:extend-probl-with} and conclusions in \cref{sec:Conclusion}.
The overall goal of this work is to make higher-order moments more accessible and usable for statistical and data analysis.

\section{Notation and background}\label{sec:Prelims} 

Lowercase letters denote scalars, e.g., $a$.
Lowercase bold letters denote vectors, e.g. $\Vc{A}$; and the $j$th element of $\Vc{a}$ is $a_j$.

Uppercase bold letters denote matrices, e.g., $\A$;
the $j$th column vector of matrix $\A$ is denoted as $\Vc{a}_j$;
and the $(j,k)$ element of $\A$ is $a_{jk}$.
The notation $\A \ast \Mx{B}$ denotes elementwise multiplication for two matrices of the same size.
The notation $[\A]^d$ denotes the elementwise power operation, i.e., each element is raised to the $d$th power.
For a vector $\Vc{a} \in \Real^n$, the notation $\Mx{D}_{\Vc{a}} \equiv \diag(\Vc{a})$, i.e., the $n \times n$ diagonal matrix with $\Vc{a}$ on its diagonal.

Uppercase bold letters in Euler script, such as $\X$, denote higher-order tensors, i.e., $d$-way arrays with $d \geq 3$.
We say $d$ is the order of the tensor.
The $(i_1,i_2,\dots,i_d)$ element of $\X$ is denoted as $\xe*$.
For notational convenience, we let $i$ denote a
\emph{multiindex} such that $i \equiv (i_1,i_2,\ldots,i_d)$, thus $\xe \equiv \xe*$.
We denote the space of real-valued tensors of order $d$ and dimension $n$ by $\Rdn$.
In storage and computational complexity analyses, we treat the tensor order, $d$, as a constant.

\subsection{Inner product and norm of tensors} \label{sec:IPNorm}
The inner product of two tensors $\X, \T{Y} \in \Real^{[d,n]}$ is the sum of the product of their corresponding entries, i.e.,
\begin{equation} \label{eq:tensorIP}
  \dotprod{\X}{\T{Y}} 
  =\bigsum \xe* \, y_{i_1 i_2\cdots i_d}.
\end{equation}
If $\Vc{x},\Vc{Y}$ are vectors, then $\dotprod{\Vc{x}}{\Vc{y}} = \Vc{x}^T \Vc{y}$.
The norm of a tensor is the square root of the sum of the squares of its entries, i.e.,
$\| \X \|^2= \dotprod{\X}{\X}$.
The cost to compute the inner product or norm is $O(n^d)$.

\subsection{Symmetric tensors}
A tensor is called \emph{symmetric} if its entries are invariant under any permutation of the indices~\cite{CoGoLiMo08}.
We let $\Sdn \subset \Rdn$ denote the subspace of symmetric tensors.
From \cite{BaKoPl11},
the number of unique entries in a $d$-way $n$-dimensional symmetric tensor is
\begin{displaymath}
\left( {{n+d-1}\atop{d}} \right) \approx \frac{n^d}{d!}.  
\end{displaymath}

\subsection{Tensor times same vector}
\label{sec:tensor-times-same}

A commonly used tensor operation is called \emph{tensor times same vector} (TTSV) \cite{KoMa11}.
The computation of the TTSV \emph{in all modes} for a tensor $\X \in \Sdn$ and vector $\Vc{a} \in \Real^n$
is denoted by $\X\Vc{a}^d$  and the result is a scalar:
\begin{equation}\label{eq:ExpXbd}
\X \Vc{a}^d = \bigsum \left( \xe* \prod_{k=1}^d a_{i_k} \right).
\end{equation}
This operation is visualized in \cref{fig:ttsv}.
\begin{figure}
  \centering
 \includegraphics[scale=0.6]{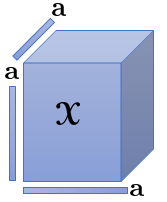}  
  \caption{Illustration of symmetric tensor times same vector (TTSV) in all modes for a three-way ($d=3$) tensor.}
  \label{fig:ttsv}
\end{figure}
The computation of the TTSV \emph{in all modes but one} is denoted $\X\Vc{a}^{d-1}\in \Real^n$ and results in a vector of size $n$:
\begin{equation}\label{eq:ExpXbdd}
  \left( \X \Vc{a}^{d-1} \right)_{i_1}
  = \bigsum{2} \left( \xe* \prod_{k=2}^d a_{i_k} \right)
  \qtext{for all} i_1 \in \set{1,\dots,n}.
\end{equation}
The choice of ``left out'' mode in the computation of all modes but one does not matter because of the symmetry.
Observe that $\X \Vc{a}^d = \dotprod{\X\Vc{a}^{d-1}}{\Vc{a}}$,
so we can use the result of the TTSV in all modes but one to compute the TTSV in all modes.
The cost to compute either version of the TTSV is $O(n^d)$.

\subsection{Symmetric outer product} \label{sec:OP}
The outer product of an $n$-dimensional vector with itself results in a rank-1 symmetric matrix of size $n \times n$ such that
\begin{equation*}
  (\Vc{a} \otimes \Vc{a})_{ij} = a_{i} a_{j}
   \qtext{for all} i,j \in \set{1,\dots,n}.
\end{equation*}
Here, the outer product is denoted by $\otimes$.
The three-way tensor outer product of an $n$-dimensional vector with itself results in a three-way symmetric tensor such that
\begin{equation*}
(\Vc{a} \otimes \Vc{a} \otimes \Vc{a})_{ijk} =a_{i}a_{j}a_{k} \qtext{for all} i,j,k \in \set{1,\dots,n}.
\end{equation*}
Generally speaking, the $d$-way tensor outer product of $\Vc{a}$ with itself $d$-times, denoted as $\Vc{a}^{\otimes d}$, results in a $d$-way rank-1 symmetric tensor such that
\begin{equation}\label{eq:outer_product}
  \Vc{a}^{\otimes d} =
  \underbrace{\,\Vc{a} \otimes \Vc{a} \otimes \ldots \otimes \Vc{a}\,}_{d \hbox{ times}} 
  \qtext{with}
  \left(\Vc{a}^{\otimes d}\right)_{i_1 i_2 \dots i_d} = \prod_{k=1}^d \Vc{a}_{i_k}
  \qtext{for all}
  i_1,\ldots,i_d \in \set{1,\dots,n}.	
\end{equation}
The cost to form $\Vc{a}^{\otimes d}$ explicitly is $O(n^d)$.

\subsection{Symmetric Kruskal tensors}
\label{sec:symm-krusk-tens}
We define $\Kdnr \subset \Sdn$ to be the space of symmetric
\emph{Kruskal} tensors \cite{BaKo07} that can be written as the sum of $r$
components where each \emph{component} is a scalar times a $d$-way
$n$-dimensional rank-one symmetric tensor.
Mathematically, 
\begin{equation} \label{eq:symmDecomp}
  \Kdnr \equiv \Set{\rsum \mterm | \lvec \in \Real^r, \A \in \Real^{n \times r}} \quad \subset \Sdn.
\end{equation}
The (symmetric) rank of a symmetric tensor is the \emph{minimal} $r$ such that the tensor can be expressed as the sum of $r$ symmetric rank-one components.
Hence, every tensor in $\Kdnr$ is of rank at most $r$.
Colloquially, $\Kdnr$ may be referred to as the set of rank-$r$ tensors even though $r$ is only an upper bound on the rank.

\subsection{Symmetric tensor decomposition}
\label{sec:symm-tens-decomp}

The goal of symmetric tensor decomposition
is to find a rank-$r$ symmetric CP decomposition $\M \in \Kdnr$ that is a good approximation of a given tensor $\X \in \Sdn$.
Ideally, we would choose a value of $r$ that is equal to the rank of $\X$;
however, there is no direct way to find or compute the rank of a symmetric tensor \cite{CoGoLiMo08}.
Indeed, computing the rank has been shown to be NP-hard \cite{HiLi13}.
For this work, we assume that the value of $r$ is specified by the user.

There are a few ways to formulate the optimization problem \cite{Ko15}, but here we simply consider
the sum of the squared errors between the entries of the approximation $\M$ and data tensor $\X$.
Thus, for a given rank $r$, computing a rank-$r$ symmetric CP decomposition reduces to solving the
nonconvex optimization problem:
\begin{equation}
  \label{eq:optFunc}
  \min_{\lvec,\A}
  f(\lvec,\A)  \equiv \| \X - \M \|^2
  \qtext{s.t.}
  \M \equiv \rsum \mterm.
\end{equation}
Computing $f(\lvec,\A)$ requires $O(rn^d)$ operations to form $\M$ and compute the difference.

We can compute the gradients and use a standard first-order optimization method to solve \cref{eq:optFunc}, though we are not
guaranteed of finding an optimum and usually employ multiple starts in the hopes of finding a good solution.
From \cite{Ko15}, the corresponding gradients are 
\begin{align}
   \label{eq:gradl} 
  \frac{\partial f}{\partial \lj} & = -2 \Biggl[ \X \Aj^d - \sum_{k=1}^r \lj{k} \; \dotprod{\Aj}{\Aj{k}}^d \Biggr] \qtext{and}\\
  \label{eq:gradb}
  \frac{\partial f}{\partial \Aj} & =-2d\lj \Biggl[ \X \Aj^{d-1} 
                                    - \sum_{k=1}^r \lj{k} \; \dotprod{\Aj}{\Aj{k}}^{d-1}\Aj{k} \Biggr],
\end{align}
for $j=1,2,\dots,r$.
The work to compute the gradient is dominated by the TTSV calculations, at a cost of $O(rn^d)$.

\section{Implicit decomposition} \label{sec:ImpDecomp}

In this paper, we focus on the case where the data tensor $\X$ has special structure.
Let $\V = \begin{bmatrix} \Vl{1} & \Vl{2} & \cdots & \Vl{p} \end{bmatrix} \in \Real^{n \times p}$
be a data matrix whose columns represent $p$ observations of a random real-valued $n$-vector.
For convenience, we define the constant vector $\nvec = \begin{bmatrix}  1/p & 1/p & \cdots & 1/p \end{bmatrix}^T \in \Real^p$.
The results we show later will apply to arbitrary positive $\nvec$.
We typically assume $n \ll p$, meaning that we have many more observations than variables  (also known as covariates or features).
As described in the introduction, the individual entries of the $d$th moment are given by \cref{eq:moment_tensor_element},
so the full tensor can be expressed in outer product notation as
\begin{equation} \label{eq:tensorX}
  \X = \psum \xterm \quad \in \Kdnp.
\end{equation}
Hence, the tensor $\X$ is a symmetric Kruskal tensor \cite{BaKo07} with $n^d$ entries, and
$\rank(\X) \leq p$.
Our goal is to show that we can reduce the computational expense from $O(rn^d)$ to $O(rnp)$ for such data tensors.
We do this by working with $\X$ implicitly, i.e., 
we can compute the function values and gradients for the rank-$r$ tensor decomposition of $\X$ using only  $\nvec \in \Real^p$, $\V \in \Real^{n \times p}$ so that
the full tensor $\X \in \Rdn$ is never formed.

\subsection{Key lemmas}
\label{sec:key-lemmas}

Our results depend on two key lemmas for symmetric rank-1 $d$-way tensors.
The first says that the inner product of two symmetric rank-1 tensors reduces to the $d$th power of the dot product of the factor vectors.
The second says that the TTSV in all modes but one has a reduced representation that is similar to that for the inner product.
The proofs are straightforward and left as an exercise for the reader.
\begin{lemma}\label{lem:rank_one_inner_product}
  Let $\Vc{a}, \Vc{b} \in \Real^n$. Then the inner product of the $d$-way symmetric rank-1  tensors constructed from the vectors satisfies
  $\big< \Vc{a}^{\otimes d}, \Vc{b}^{\otimes d} \big> = \big< \Vc{a}, \Vc{b} \big>^d$.
\end{lemma}
\begin{lemma}\label{lem:rank_one_ttsv}
  Let $\Vc{a}, \Vc{b} \in \Real^n$. Let $\T{S}$ be the symmetric rank-1
  tensor defined by $\T{S} \equiv \Vc{b}^{\otimes d}$. Then the TTSV in all modes but one of tensor $\T{s}$ with the vector $\Vc{a}$ satisfies
  $\T{S} \Vc{a}^{d-1} = \big< \Vc{a}, \Vc{b} \big>^{d-1} \Vc{b}$. Further, the TTSV in all modes satisfies $\T{S} \Vc{a}^d = \big< \Vc{a}, \Vc{b} \big>^d$.
\end{lemma}

\subsection{Computing TTSV for the gradient implicitly}
\label{sec:comp-ttsv-impl}

The most expensive computations in the gradient evaluations are the TTSVs.
\Cref{lem:ttsv} explains how to compute the TTSV for $(d-1)$ vectors using the constituent elements of $\X$.
We remind the reader that $\Dnvec = \diag(\nvec)$.

\begin{lemma}\label{lem:ttsv}
  Let $\X = \psum \xterm \in \Kdnp$ and $\Vc{a} \in \Real^n$. Then
  $\X \Vc{a}^{d-1} = \V \Dnvec [\V^T \Vc{a}]^{d-1}$.
\end{lemma}
\begin{proof}
  Using the definition of $\X$ and \cref{lem:rank_one_ttsv}, we have
  \begin{displaymath}
    \X \Vc{a}^{d-1} = \psum (\xterm)\,\Vc{a}^{d-1} = \psum \nl \; \dotprod{\Vl}{\Vc{a}}^{(d-1)} \Vl.
  \end{displaymath}
\end{proof}

The standard calculation of the TTSV is $O(n^d)$ work.
The dominant cost in the implicit version is the matrix-vector multiplies involving $\V$ for a cost of $O(np)$.
If we define the matrix $\Y \in \Real^{n \times r}$ such that $\Yj \equiv \X\Aj^{d-1}$, then we can calculate
all the TTSVs implicitly using
\begin{equation}
  \label{eq:Y}
  \Y \equiv \V \Dnvec [\V^T \A]^{d-1},
\end{equation}
for a cost of $O(npr)$ operations. The reduction from $O(n^d)$ to $O(npr)$ work is the cornerstone of the improvements yielded by the implicit method.

\subsection{Computing the function value implicitly}
\label{sec:comp-funct-value}

To compute the function value \cref{eq:optFunc}, we first rewrite it as:
$\| \X - \M \|^2 = \|\X\|^2+\|\M\|^2-2 \dotprod{\X}{\M}$.
The first term is a constant and can be ignored.
We compute the remaining two terms implicitly without forming $\X$, as described in the lemmas that follow.

\begin{lemma}\label{lem:norms}
  Let $\M = \rsum \mterm \in \Kdnr$.
  Then
  $\| \M \|^2 = \lvec^T [\A^T \A]^d \lvec$.    
\end{lemma}
\begin{proof}
  Using the definition of $\M$, rearranging terms, and applying \cref{lem:rank_one_inner_product}, we have
  \begin{displaymath}
    \big< \M, \M \big>
    = \left< \rsum \mterm, \rsum \mterm \right> 
    = \rsum \sum_{k=1}^r \lj \lj{k} \big< \Aj^{\otimes d}, \Aj{k}^{\otimes d} \big>
    = \rsum \sum_{k=1}^r \lj \lj{k} \big< \Aj, \Aj{k} \big>^d.
  \end{displaymath}
\end{proof}

The cost to form $\M$ is $O(rn^d)$, and 
the cost of computing its norm is $O(n^d)$.
Exploiting its structure using \cref{lem:norms} reduces the cost to $O(nr^2)$.
Although $\|\X\|^2$ is technically not necessary (since it is a constant term),
we include it here for completeness.
\Cref{lem:norms} applied to $\X$ yields $\| \X \|^2 = \nvec^T (\V^T\V)^d \nvec$, in which case the cost to compute the norm is $O(np^2)$.

\begin{lemma}
  Let $\M = \rsum \mterm \in \Kdnr$.
  Then
  $\dotprod{\X}{\M} = \Vc{w}^T \lvec$
  where $w_j = \X \Aj^d$ for  $j \in \set{1,\dots,r}$.
\end{lemma}

The proof is straightforward and left to the reader.
If $\Vc{w}$ is precomputed, the cost of the dot product is only $O(r)$.
We can compute $\Vc{w}$ using $\Y$ from \cref{eq:Y} as $w_j = \Yj^T \Aj$ for $j=1,\dots,r$ for a cost of $O(nr)$ operations.

\subsection{Implicit versus explicit algorithms}
\label{sec:impl-vers-expl}

Putting the information above together, we have the function and gradient evaluation algorithms for both the explicit and implicit cases in \cref{fig:algs}.
Recall that we are computing $f(\lvec,\A) \equiv \|\X - \M\|^2$ where $\M = \rsum \mterm$,
and the function and gradient are defined in \cref{eq:gradl,eq:gradb}.
The function and gradient can be used in any first-order optimization method to compute the best rank-$r$ model.\footnote{The problem is non-convex, so there is no guarantee of finding a global minimum. In practice, one usually obtains good results with a few random starting points.}

For the explicit case in \cref{alg:explicit}, the inputs are the (explicit) moment tensor $\X$,
the values of $\lvec \in \Real^r$ and $\A \in \Real^{n \times r}$ (used to form $\M$),
and $\alpha = \|\X\|^2$ which is a constant that can be ignored by setting it equal to zero.
The outputs are the function value ($f$) and gradients.
We let $\gl \in \Real^r$ denote the gradient with respect to $\lambda$
whose entries are defined by~\cref{eq:gradl}.
Likewise, we let $\gA \in \Real^{n \times r}$ denote the gradient with respect to $\A$;
the columns of $\gA$ are defined by \cref{eq:gradb}.

For the implicit case in \cref{alg:implicit}, the input data tensor $\X$
is replaced by its constituent parts, i.e., $\nvec \in \Real^p$ and $\V \in \Real^{n \times p}$.
Recall that $\nvec$ is the constant vector whose entries are $1/p$.
The implicit approach avoids explicitly forming the moment tensor, which would have an additional one-time cost of $O(pn^d)$
in addition to the extra computational cost in the function and gradient evaluation.
The differences in the explicit algorithm are highlighted in red.

In both algorithms, we define several new variables for efficient reuse of computations,
namely $\Mx{B} \equiv \A^T \A$, $\Mx{C} \equiv [\A^T\A]^{d-1}$ and
$\Vc{u} \equiv [\A^T\A]^d \lvec$.

Because we make maximal reuse of computations, the sole difference between the explicit and implicit method is in the computation of the matrix $\Y$, representing the TTSVs. Assuming $p,n \gg r$, computation of $\Y$ is the dominant cost of either method.
Lines 3--9 cost $O(nr^2)$.
The cost of computing $\Y$ (in Line 2) is $O(rn^d)$ for the explicit version as compared to $O(pnr)$ for the implicit version.
Hence, the implicit version is less expensive so long as $p \ll n^{d-1}$.

In terms of memory, both explicit and implicit have the same requirement for the outputs and temporaries:  $O(nr)$ storage.
For the explicit case, the input tensor $\X$ requires $O(n^d)$ storage.
For the implicit case, the inputs $\nvec$ and $\V$ require $O(np)$ storage.

The one-time optional cost to compute the input $\alpha = \|\X\|^2$ is $O(n^d)$ in the explicit case and $O(np^2)$ in the implicit case. For some larger values of $p$, the implicit cost may be more expensive than the explicit formation of $\X$. 

\begin{figure}
  \begin{minipage}{0.475\linewidth}
\begin{algorithm}[H]\small
  \caption{Explicit computation} \label{alg:explicit}
  \begin{algorithmic}[1]
    \Function{fg\_explicit}{${\X},\lvec, \A,\alpha$} %
    \LineFor{$j = 1,\dots,r$}{$\Vc{Y}_j=\X\Aj^{d-1}$}
    \LineFor{$j=1,\dots,r$}{$w_j=\Aj^T \Vc{y}_j$}
    \State $\Mx{B} = \A^T \A$
    \State $\Mx{C} = [\Mx{B}]^{d-1}$ %
    \State $\Vc{u} = (\Mx{B} \ast \Mx{C}) \lvec$ %
    \State $f= \alpha + \lvec^T  \Vc{u} - 2\Vc{w}^T \lvec$
    \State $\gl= -2 ( \Vc{w} - \Vc{u})$
    \State $\gA = -2 d ( \Y - \A \Dlvec \Mx{C} ) \Dlvec$ 
    \State \Return $f, \gl, \gA$
    \EndFunction
  \end{algorithmic}
\end{algorithm}%
\end{minipage}\hfil
  \begin{minipage}{0.475\linewidth}
\begin{algorithm}[H]\small
  \caption{Implicit computation} \label{alg:implicit}
  \begin{algorithmic}[1]
    \Function{fg\_implicit}{$\textcolor{red}{\nu,\V}, \lvec,\A,\alpha$}
    \State \textcolor{red}{$\Y = \V \Dnvec  [\V^T\A]^{d-1}$}
    \LineFor{$j=1,\dots,r$}{$w_j=\Aj^T \Vc{y}_j$}
    \State $\Mx{B} = \A^T \A$
    \State $\Mx{C} = [\Mx{B}]^{d-1}$ %
    \State $\Vc{u} = (\Mx{B} \ast \Mx{C}) \lvec$ %
    \State $f= \alpha + \lvec^T  \Vc{u} - 2\Vc{w}^T \lvec$
    \State $\gl= -2 ( \Vc{w} - \Vc{u})$
    \State $\gA = -2 d ( \Y - \A \Dlvec \Mx{C} ) \Dlvec$ 
    \State \Return $f, \gl, \gA$
    \EndFunction
  \end{algorithmic}
\end{algorithm}%
\end{minipage}
\caption{Algorithms for computation of the function and gradient for symmetric CP decomposition with an explicit data tensor $\X$ versus implicit computation using $\nvec$ and $\V$ where $\X = \psum \xterm$.
  Differences are shown in red.
  Here, the input $\alpha = \|\X\|^2$ is a constant term in the function that can be set to anything without impacting the optimization.
  Recall that powers outside bracketed matrices are elementwise powers.
}
\label{fig:algs}
\end{figure}

\section{Numerical results} \label{sec:Results}
In this section, we present results from numerical experiments to demonstrate
the effectiveness of the implicit approach for moment tensors.
All experiments were run on a dedicated computer with a dual socket Intel E5-2683v5 2.00GHz CPU (28 total cores) with 256 GB DDR3 memory, using MATLAB 2018a. No specific parallel coding was used, but some calls within MATLAB take advantage of multiple cores.
All methods are implemented using the Tensor Toolbox \cite{TensorToolbox}.
We use a standard first-order optimization: limited-memory BFGS with bound constraints (L-BFGS-B) \cite{ZhByLuNo97} 
as the optimization method
for the explicit and implicit methods, using the implementation provided by Stephen Becker's MATLAB wrapper \cite{BeckerLBFGSB}.
We use no lower or upper bounds, i.e., we set them to $-\infty$ and $+\infty$, respectively.
We use the default options (including the default memory of $m=5$), with the following exceptions: maximum iterations = 10,000; maximum total iterations = 50,000; and the gradient norm tolerance (\texttt{pgtol}) was set as specified in the subsections that follow.

\subsection{Timing comparison of explicit and implicit methods}
\label{sec:timing-comp-expl}

We compare the timings for the explicit and implicit methods for different scenarios depending on $d$~(order), $n$~(size), $p$~(number of samples), and $r$~(rank of approximation).
For each scenario,
the observation matrix $\V \in \Real^{n \times p}$ has entries from the uniform $(0,1)$ distribution, so we are not expecting the low-rank approximation to have any specific structure.
Since no structure is expected, we set the optimization tolerance to be relatively low, i.e., $\texttt{pgtol} = .05$.
For each scenario, we run the optimization ten times with ten different initial guesses and report average results.

The results are shown in \cref{tab:timing_compare}.
We report the mean and standard deviation of the total run time and number of iterations for the explicit and implicit methods.
The number of iterations is the total number of \emph{inner} iterations, i.e.,
the number of times that the function/gradient evaluation is called.
The table includes the average time per iteration, as computed for the data from all ten runs.
The best relative error for each scenario is not reported in the table but is in the range $(1 \times 10^{-3}$, $8 \times 10^{-3})$.
We also omit the time to create the explicit moment tensor, but it is always less than 10\% of the optimization time.
Before we dig into the results, we explain the variations in the number of iterations.
Given the same initial guess, 
the explicit and implicit functions and gradients are mathematically identical; however, small differences may be introduced because of numerical round-off errors in the different computations.
In the context of an optimization method, these small differences can result in different paths to the solution or even different solutions, as well as different numbers of iterations.
Nevertheless, the absolute difference in the final objective value
between any paired explicit and implicit run (i.e., with the same starting point) is $2 \times 10^{-5}$.

\begin{table}
  \centering\footnotesize
  \begin{tabular}{|*{4}{r|}r|r|*{4}{r@{$\pm$}l|}}
    \hline
\multicolumn{4}{|c|}{Scenario} & \multicolumn{2}{c|}{Time Per Iter.~(s)} & \multicolumn{4}{c|}{Total Time (s)} & \multicolumn{4}{c|}{Total Iterations} \\ 
\multicolumn{1}{|c|}{$d$} & \mchm{n} & \mchm{p} & \mchm{r} & \mch{Exp.} &\mch{Imp.} & \multicolumn{2}{c|}{Exp.} & \multicolumn{2}{c|}{Imp.} & \multicolumn{2}{c|}{Exp.} & \multicolumn{2}{c|}{Imp.} \\ \hline
3  &  75 &  750 &  5 & 5.30e-04 & \bf 4.36e-04 &   1.84 &  0.52 &   1.64 &  0.39 &  3473 &  959 &  3767 &  955 \\ \hline
3  &  75 &  750 & 25 & 1.91e-03 & \bf 5.91e-04 &  17.79 &  3.34 &   5.09 &  1.05 &  9296 & 1650 &  8606 & 1788 \\ \hline
3  &  75 & 3750 &  5 & \bf 5.17e-04 & 7.96e-04 &   2.08 &  0.63 &   3.11 &  1.08 &  4023 & 1249 &  3905 & 1375 \\ \hline
3  &  75 & 3750 & 25 & 1.90e-03 & \bf 1.03e-03 &  15.96 &  2.76 &   8.78 &  0.91 &  8384 & 1452 &  8504 &  854 \\ \hline
3  & 375 & 3750 &  5 & 2.29e-02 & \bf 4.82e-03 & 240.45 & 45.60 &  43.77 & 11.10 & 10506 & 1951 &  9075 & 2488 \\ \hline
4  &  75 & 3750 &  5 & 1.47e-02 & \bf 9.22e-04 &  75.64 & 20.89 &   4.16 &  1.06 &  5160 & 1438 &  4514 & 1081 \\ \hline
  \end{tabular}
  \caption{Timing comparison of explicit and implicit methods with different scenarios for $d$ (order), $n$ (size), $p$ (\# observations) and $r$ (rank), reporting average results for 10 optimization runs for each scenario.
    The moment tensor comes from the situation where the $p$ observations are vectors of length $n$ drawn from the uniform distribution.
    All times are reported in seconds.
    ``Time Per Iter.'' is the average time per inner iteration over all runs.
    ``Total Time'' is the overall optimization time, reporting mean and standard deviation for the 10 runs.
    ``Total Iterations'' is the total number of inner iterations (each of which requires a function/gradient calculation),
    reporting mean and standard deviation for the 10 runs.
  }
  \label{tab:timing_compare}
\end{table}

Our main focus is on how the cost per inner iteration (i.e., per function/gradient evaluation) grows under different scenarios. Recall that the cost per (inner) iteration is $O(rn^d)$ for the explicit method as compared to $O(pnr)$ for the implicit method. Hence, we expect the speedup of the implicit method to be $O(n^{d-1}/p)$. 
{There is little difference between the explicit and implicit methods when $n^d$ is relatively small.}
On current computing architectures, ``relatively small'' equates to $n^d < 10^6$.
For the scenarios where $d=3$ and $n=75$, the explicit 3rd order moment tensor is less than 4MB in size.
With one exception, the implicit method is slightly faster per iteration than the explicit method.
The one exception is $[d=3, n=75, p=3750, r=5]$, which is not surprising since $p \approx n^{d-1}$.
{There is a 5${\times}$ or greater improvement for the implicit method when $n^d$ is relatively large.}
  On current computing architectures, ``relatively large'' equates to $n^d > 10^7$.
  For the scenarios with $[d=3,n=375]$ and  $[d=4,n=75]$, the storage for the explicit moment tenors are 421~MB and 323~MB
  and the implicit method is faster by factors of 5 and 16, respectively.
{The per-iteration cost of the explicit method is insensitive to $p$.} Increasing $p$ from 750 to 3750 for $[d=3,n=75]$ causes no change in the per iteration time for the explicit method, but the implicit method becomes more expensive by a factor of 2.
Conversely, {the per-iteration cost of the implicit method increases only linearly with $n$ whereas the cost of the explicit increases at a polynomial rate.} When $n$ is increased from 75 to 375 (5${\times}$), the per iteration cost of the explicit method increase by a factor of 44 versus only 6 for the implicit method.
{The per-iteration cost of the explicit method goes up exponentially in $d$.}
  As we increase $d$ from $d=3$ to $d=4$ for $[n=75,p=3750]$,
  the per iteration cost of the implicit method is nearly unchanged whereas the explicit method goes up by a factor of 28.
  
\subsection{Numerical estimation of means in spherical Gaussian mixture model}
\label{sec:numer-estim-rank}

In the examples in the previous section, the test problems had no underlying low-rank structure.
Here we demonstrate the performance on problems with known solutions.
Specifically, we consider the case of data that comes from a spherical Gaussian mixture model as has been considered, e.g., in \cite{AnGeHsKa14a}.
Each observation is of the form
\begin{displaymath}
  \Vl \sim \mathcal{N}(\Vc{\mu}_{j_\ell},\sigma^2\mathbf{I})
  \qtext{where} \text{Prob}(j_\ell = j)= 1/r \text{ for } j = 1,\dots,r.
\end{displaymath}
In other words, there are $r$ spherical Gaussians mixed together (with equal probabilities) with means $\Aji$ for $j=1,\dots,r$.
Each observation $\Vl$ is drawn from \emph{one} of the $r$ Gaussians.
If we let $V$ be the random variable corresponding to observations $\Vl$, then
\begin{displaymath}
  \mathbb{E}( V^{\otimes d} ) \approx \frac{1}{r} \sum_{j=1}^r \Vc{\mu}_j^{\otimes d} .
\end{displaymath}
The expectation is not exact because there are lower-order terms depending on $\sigma^2$;
see, e.g. Hsu and Kakade \cite{HsKa13} for the case of $d=3$.
We can potentially incorporate the lower-order terms into the implicit framework,
and we leave this as a topic for future work since it involves more extensive derivations
and is not necessary for successful recovery of the $\Vc{\mu}_j$'s.
Ignoring the lower-order terms, the $d$th-order moment has rank $r$
and can be approximated with a rank-$r$ symmetric tensor factorization such that  $\A_{\rm ideal} =
\begin{bmatrix}
  \Aji{1} & \cdots & \Aji{r}
\end{bmatrix}$
is its factor matrix
(appropriately scaled).
Unlike \cite{AnGeHsKa14a,AnGeHsKa14}, we make no assumption that $\A_{\rm ideal}$ is orthogonal and do not need to apply whitening.
We choose an example where the means $\Aji$ are correlated to show that it works even in the non-orthogonal and non-whitened case, which is generally considered to be both more difficult and more realistic.
Specifically, we require $\|\Aji\|=1$ for all $j$ and $\Aji{j_1}^T\Aji{j_2} = 0.5$ for all $j_1 \neq j_2$, which corresponds to an angle of $60^\circ$.\footnote{The function \texttt{matrandcong} from the Tensor Toolbox from MATLAB can be used to create such a matrix, which follows the method proposed by Tomasi and Bro \cite{ToBr06}.}
An example of a Gaussian mixture from $r=5$ components is shown in \cref{fig:gmm-picture}. The data is randomly projected from $n=500$ dimensions down to two. We can see that the data is highly overlapping, making discovery of the means potentially challenging.

\begin{figure}
\definecolor{mycolor1}{rgb}{0.0000,0.4470,0.7410}
\definecolor{mycolor2}{rgb}{0.8500,0.3250,0.0980}
\definecolor{mycolor3}{rgb}{0.9290,0.6940,0.1250}
\definecolor{mycolor4}{rgb}{0.4940,0.1840,0.5560}
\definecolor{mycolor5}{rgb}{0.4660,0.6740,0.1880}
\definecolor{mycolor6}{rgb}{0.3010,0.7450,0.9330}
\definecolor{mycolor7}{rgb}{0.6350,0.0780,0.1840}  
  \pgfplotsset{
  tick label style={font=\tiny},
  label style={font=\footnotesize},
  legend style={font=\small},
  title style={font=\small, at={(0.5,1.0)}},
  scale only axis,
  height = 3in,
  width = 4.5in,
  scatter,scatter src=explicit symbolic, 
  scatter/classes={
    c1={mycolor1},
    c2={mycolor2},
    c3={mycolor3},
    c4={mycolor4},
    c5={mycolor5},
    m={black}},  
  }
  \centering
  \begin{tikzpicture}
    \begin{axis}[xmin=-6, xmax=5, ymin=-6, ymax=4]
      \addplot[only marks, mark size=0.4pt] table[meta=meta] {coords.dat};
      \addplot[only marks, mark=x, mark size=5pt, line width=3pt] table[meta=meta] {truemeans.dat};
      \legend{data ($\Vl$) colored by component,,,,,,mean ($\Vc{\mu}_j$) colored by component,,,,}
     \end{axis}
  \end{tikzpicture}
  \caption{Example of data from Gaussian mixture model for $n=500$ with $r=5$ components of the form $\mathcal{N}(\Aji,\sigma^2\Mx{I})$.
    The component means are such $\|\Aji\|=1$ for all $j$ and  $\Aji{j_1}^T\Aji{j_2} = 0.5$ for all $j_1 \neq j_2$, and the covariance is $\sigma^2 \Mx{I}$ with
    $\sigma = 10^{-1}$. We show $p=750$ samples, distributed uniformly across the five components. The samples are color-coded according to component. The data has been randomly projected to two dimensions for the purposes of visualization.}
  \label{fig:gmm-picture}
\end{figure}

Results of a series of experiments with Gaussian mixture models are shown in \cref{fig:gmm3,fig:gmm4} for $d=3$ and $d=4$, respectively.
For all experiments, we fix $n=500$.
We only use the implicit method since the explicit method would be too expensive in time and memory.
For $d=3$, the explicit moment tensor would be 1~GB in size; for $d=4$, it would be 500~GB.
We experiment with different numbers of components ($r \in \set{3,5,10}$)
and different noise levels ($\sigma \in \set{10^{-4},10^{-3},10^{-2},10^{-1}}$).
For each scenario (i.e., $d$,$r$,$\sigma$), we generate $p=250r$ observations, i.e., 250 samples per component.
One generally does not know the true rank, so 
we test different values for the  rank of the approximation in the range $\hat r \in \set{r-2,\dots,r+2}$.
We set the convergence tolerance to be \texttt{pgtol} = 1e-4.
For each value of $\hat r$, we run the optimization procedure ten times and take the best solution, i.e., with the lowest final function value ($f$).

Before we discuss our detailed results, we first discuss 
the choice of random initialization for the optimization method, which 
is critical to the success of the method.
If $r \ll n$, there is a high probability that a random initial guess such as $\A_0 \sim \mathcal{N}(0,1)^{n \times \hat r}$ (with columns normalized)
is nearly orthogonal to the means and the resulting observations,
i.e., $\A_0^T\A_{\rm ideal} \approx 0$ and $\A_0^T\V \approx 0$.
To avoid this dilemma, we propose initialization based on the randomized range finder (RRF) \cite{HaMaTr11}:
\begin{equation}
  \label{eq:rrf}
  \Mx{A}_0 = \V \Mx{\Omega}
  \qtext{(with columns normalized) where}
  \Mx{\Omega}\sim\mathcal{N}(0,1)^{p \times \hat r}.
\end{equation}
The matrix $\Mx{\Omega}$ is called a Gaussian testing matrix.
We illustrate the advantage of the RRF initialization in \cref{fig:init-compare} for
an exemplar scenario: $d=3$, $r=3$, $\hat r=3$, and $\sigma \in \set{10^{-4},10^{-3},10^{-2},10^{-1}}$.
The results for other values of $d$, $r$ and $\hat r$ are similar.
For each type of initialization, 
we do 100 optimization runs, each time with a different starting guess.
We select the best function value over \emph{all} runs (100 per initialization method, for a total of 300 runs) for each value of $\sigma$, and then report how often that initialization technique converged to a value within 0.01 of the best value.
We also report the total time for the 100 optimization runs at the top of each bar.
As we see in  \cref{fig:init-compare}, the RRF initialization is much more successful than the random initialization in terms of the number of successful optimization runs. 
As  mentioned above,
if $\A_0$ is close to orthogonal to $\A_{\rm ideal}$ (which is expected if it is random),
then it is also close to orthogonal to the random observations in $\V$ for smaller values of $\sigma$.
As a result, the values in  $[\V^T\A]^{d-1}$ are very small,
leading to a very small gradient
and premature or at least slow convergence of the optimization method.
The issue is less pronounced for larger values of $\sigma$ since the additional noise lowers the probability of being nearly orthogonal.
To understand the effect of premature convergence, we tighten the convergence tolerance (\texttt{pgtol=1e-6}) in the experiment to show that the success rate is higher in those cases. However, the optimization now takes about 2--3 times longer than with the RRF initialization.

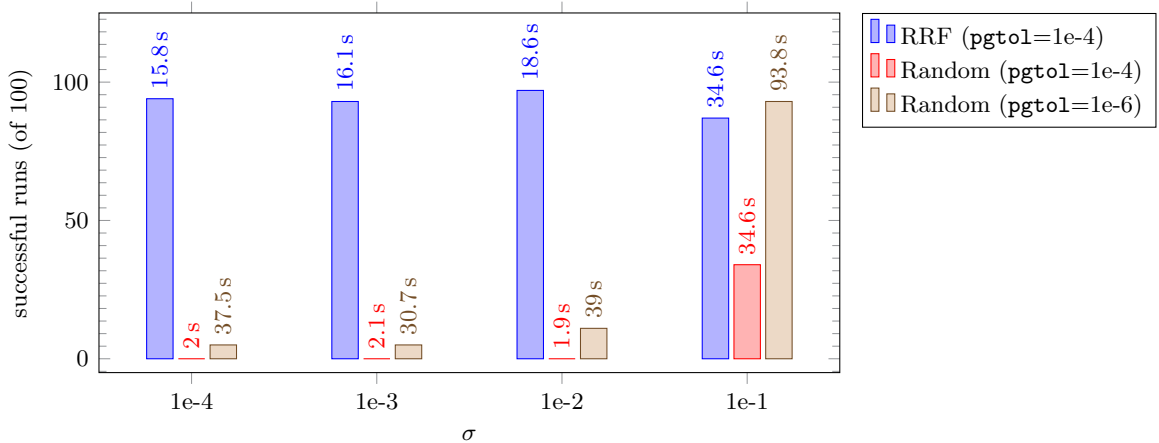
\begin{figure}
  \pgfplotsset{
  }%
  \centering
  \pgfplotstableread{init.dat}\inittable
  \begin{tikzpicture}
    \begin{axis}[
      width=4.5 in, height=2.5in,
      ybar, bar width=10pt,
      ymin=-5, ymax=125,
      ytick distance=50,
      minor y tick num=10,
      ylabel=successful runs (of 100),
      xmin=0.5,xmax=4.5, xtick={1,2,3,4},
      xticklabels={1e-4, 1e-3, 1e-2, 1e-1}, xlabel=$\sigma$,
      tick label style={font=\footnotesize},
      legend style={cells={anchor=west}, legend pos=outer north east,
        font=\footnotesize},
      point meta=explicit,
      nodes near coords={$\pgfmathprintnumber{\pgfplotspointmeta} \, \si{\second}$},,
      every node near coord/.append style={
        anchor= west, rotate=90, font=\footnotesize},
      label style={font=\footnotesize},
      ]
      \addplot table[x=Index,y=RRF4_cnts,meta=RRF4_times] {\inittable};
      \addplot table[x=Index,y=G4_cnts,meta=G4_times] {\inittable};
      \addplot table[x=Index,y=G6_cnts,meta=G6_times] {\inittable};
      \legend{RRF (\texttt{pgtol}=1e-4), Random (\texttt{pgtol}=1e-4), Random (\texttt{pgtol}=1e-6), Random (\texttt{pgtol}=1e-8), Random (\texttt{pgtol}=1e-10)}
    \end{axis}
  \end{tikzpicture}
  \caption{Successful optimization runs (out of 100) for different choices for the initial guess $\A_0$. We compute a rank $\hat r=3$ approximation for moment tensors of order $d=3$. Data are generated by a spherical Gaussian mixture model with $r=3$ components, $p=750$ samples proportionally divided among them. The observations are of size $n=500$ and the noise is defined by $\sigma$.
    We compare the RRF, i.e., $\A_0 = \V\Mx{\Omega}$ for $\Mx{\Omega} \sim \mathcal{N}(0,1)^{p \times \hat r}$, and random, i.e., $\A_0 \sim \mathcal{N}(0,1)^{n \times \hat r}$, with  different convergence tolerances (\texttt{pgtol)}.  A run is considered successful if the final solution is within 0.01 of the best over all runs using any method for that value of $\sigma$, and the vast majority of RRF r\textsc{section 5}uns are successful in every case. Each bar is capped with the total time for the 100 runs (in seconds), showing that tightening the convergence threshold can improve the success rate at a cost of a longer run time.}
  \label{fig:init-compare}
\end{figure}

\pgfplotstableread{gmm.dat}\gmmtable
\begin{figure}%
\pgfplotsset{
  tick label style={font=\footnotesize},
  label style={font=\footnotesize},
  legend style={font=\small},
  title style={font=\small, at={(0.5,1.0)}},
  scale only axis,
  width = 1.3in,
  xtick distance = 1,
}%
\centering%
\begin{tabular}{rrr}
  {\RelErrPlot{3}{3}} &
  {\RelErrPlot{3}{5}} &
  {\RelErrPlot{3}{10}}
  \\
  {\ScorePlot{3}{3}} & 
  {\ScorePlot{3}{5}} &
  {\ScorePlot{3}{10}}
  \\
  {\TimePlot{3}{3}{15}} & 
  {\TimePlot{3}{5}{80}} &
  {\TimePlot{3}{10}{525}}  
\end{tabular}~~~ \\[2mm]
\pgfplotslegendfromname{relegend33}
  \caption{
    We compute a rank-$\hat r$ approximation for moment tensors of order $\bm{d=3}$.
    Data are generated by a spherical Gaussian mixture model with $r$ components and $p$ samples proportionally divided among them.
    The observations are of size $n=500$ and the noise is defined by $\sigma$.
       The top plot is the smallest relative error over ten optimization runs.
    The middle plot is the corresponding similarity score of the low-rank model factors and the true means.
    The bottom plot is the total time for the ten optimization runs.
    The true rank ($r$) is shown by a vertical dashed line.}
  \label{fig:gmm3}
\end{figure}
\begin{figure}%
\pgfplotsset{
  tick label style={font=\footnotesize},
  label style={font=\footnotesize},
  legend style={font=\small},
  title style={font=\small, at={(0.5,1.0)}},
  scale only axis,
  width = 1.3in,
  xtick distance = 1,
}%
\centering%
\begin{tabular}{rrr}
  {\RelErrPlot{4}{3}} &
  {\RelErrPlot{4}{5}} &
  {\RelErrPlot{4}{10}}
  \\
  {\ScorePlot{4}{3}} & 
  {\ScorePlot{4}{5}} &
  {\ScorePlot{4}{10}}
  \\
  {\TimePlot{4}{3}{15}} & 
  {\TimePlot{4}{5}{80}} &
  {\TimePlot{4}{10}{525}}  
\end{tabular}~~~ \\[2mm]
\begin{tikzpicture}
  \pgfplotslegendfromname{relegend33}
\end{tikzpicture}
  \caption{
    We compute a rank-$\hat r$ approximation for moment tensors of order $\bm{d=4}$.
    Data are generated by a spherical Gaussian mixture model with $r$ components and $p$ samples proportionally divided among them.
    The observations are of size $n=500$ and the noise is defined by $\sigma$.
       The top plot is the smallest relative error over ten optimization runs.
    The middle plot is the corresponding similarity score of the low-rank model factors and the true means.
    The bottom plot is the total time for the ten optimization runs.
    The true rank ($r$) is shown by a vertical dashed line.}
  \label{fig:gmm4}
\end{figure}

The results henceforth use the RRF initialization.
In \cref{fig:gmm3,fig:gmm4}, plots in the top row show the final relative error for the best of ten runs, each with a different random initialization.
(Our experience suggested that ten random initializations was sufficient to find a good solution because repeated runs converged to within 0.01 of the best solution an average of 8 times out of 10.)
Note that the error is calculated based only on observed data using the methodology in \cref{sec:comp-funct-value}.
As the rank increases, the relative error will eventually stagnate and the rank might be estimated as the first rank before stagnation.
For $\sigma \leq 10^{-2}$,
the relative error drops significantly at $\hat r = r$ and remains fairly constant for $\hat r > r$. 
For $\sigma=10^{-1}$, the difference in errors across the different ranks is barely noticeable, which makes determining the rank difficult.
Nevertheless, the method is still successful at recovering the true means, even if the rank is overestimated, as will be seen in the middle plots.
In terms of the final error, there is little difference between $d=3$ and $d=4$.

The middle plots show the recovery of the true means, using the similarity score for the solution with the lowest final function value.
Without loss of generality, we assume $\|\Aji\| = 1$ for $j=1,\dots,r$ and $\|\Aj\|=1$ for $j=1,\dots,\hat r$.
The score is the average cosine of the angles between the computed means in $\A$ and the true means in $\A_{\rm true}$:
\begin{equation}\label{eq:Score}
  \text{score} =
  \begin{cases}
    \displaystyle\max_{\pi \in \Pi(\hat r,r)} \frac{1}{\hat r} \sum_{j=1}^{\hat r}  |\Aji{\pi(j)}^T\Aj|
    & \text{if } \hat r < r, \text{ and}\\
    \displaystyle\max_{\pi \in \Pi(r,\hat r)} \frac{1}{r} \sum_{j=1}^r |\Aji^T\Aj{\pi(j)}|
    & \text{if } \hat r \geq r,\\
  \end{cases}
\end{equation}
where
$\Pi(r_1, r_2) = \Set{ \pi: \set{1,\dots,r_1} \rightarrow \set{1,\dots,r_2} | \pi(j_1) \neq \pi(j_2) \forall j_1 \neq j_2 }$
is the set of all possible one-to-one mappings for $r_1 \leq r_2$.
Ideally, the score should be 1.
If $\hat r < r$, then $\A$ can only match a subset of the true means. Nonetheless, we can still check if a subset of the means is correctly identified.
Likewise, if $r > \hat r$, there are a surplus of candidates to match the means, so we pick the best ones.
Here, we see the advantage of higher-order moments because the results for $d=4$ are superior to those for $d=3$.
Improvements with increasing dimensionality have been observed in similar situations \cite{AnBeGoRa14}.

The bottom plots show the sum of the run times for all ten optimization
runs. Note that the y-axis scales are different for each value of $r$
but held constant across the two plots (for $d=3,4$) for easy
comparison. Observe that the time for the higher order $d=4$ is actually less than that of the lower order $d=3$.

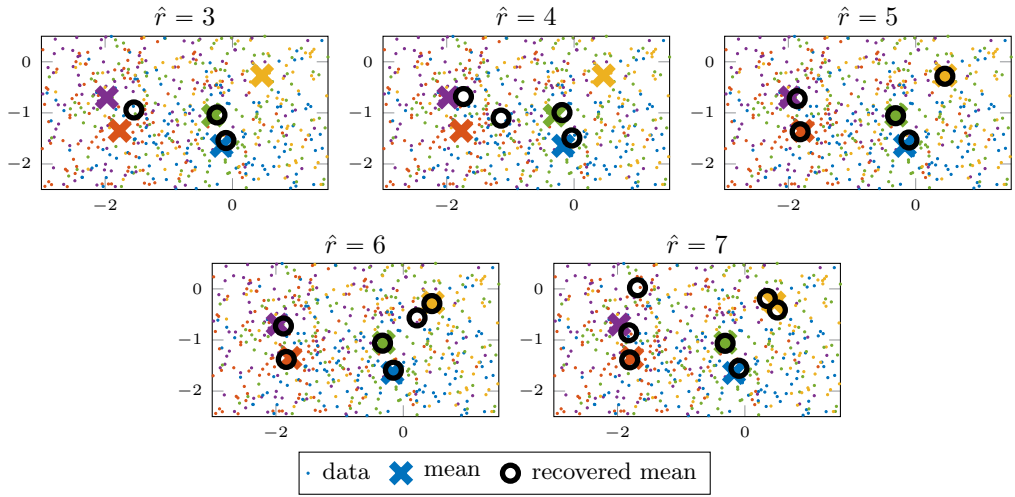
\begin{figure}
\definecolor{mycolor1}{rgb}{0.0000,0.4470,0.7410}
\definecolor{mycolor2}{rgb}{0.8500,0.3250,0.0980}
\definecolor{mycolor3}{rgb}{0.9290,0.6940,0.1250}
\definecolor{mycolor4}{rgb}{0.4940,0.1840,0.5560}
\definecolor{mycolor5}{rgb}{0.4660,0.6740,0.1880}
\definecolor{mycolor6}{rgb}{0.3010,0.7450,0.9330}
\definecolor{mycolor7}{rgb}{0.6350,0.0780,0.1840}  
  \pgfplotsset{
  tick label style={font=\tiny},
  label style={font=\footnotesize},
  legend style={font=\footnotesize},
  title style={font=\small, at={(0.5,0.9)}},
  scale only axis,
  height = .8in,
  width = 1.5in,
  scatter,scatter src=explicit symbolic, 
  scatter/classes={
    c1={mycolor1},
    c2={mycolor2},
    c3={mycolor3},
    c4={mycolor4},
    c5={mycolor5},
    m={black,mark=o, mark size = 3pt, line width=2pt}},  
  legend cell align={left}
}
  \centering
  \RecoveryPlot{3} 
  \RecoveryPlot{4} 
  \RecoveryPlot{5}\\    
  \RecoveryPlot{6} 
  \RecoveryPlot{7}\\ 
  \begin{tikzpicture}
  \pgfplotslegendfromname{mlegend5}    
  \end{tikzpicture}
  \caption{For the problem described in \cref{fig:gmm-picture}, recovery of the means from the third-order ($d=3$) moment tensor. The true rank is $r=5$. The rank of the approximation is in the range $\hat r \in \set{3,\dots,7}$. The solution shown corresponds to the lowest relative error over ten runs. The recovered means (columns $\Aj$ for $j=1,\dots,\hat r$) are shown as black circles and should, ideally, align with the true means ($\Aji$ for $j=1,\dots,r$) given by colored x's. The random projection is the same as was used in \cref{fig:gmm-picture}, but the image is zoomed in to clearly see the means.}
  \label{fig:gmm-solns}
\end{figure}

Finally, \cref{fig:gmm-solns} shows recovered means for a scenario illustrated in \cref{fig:gmm-picture}: $n=500$, $r=5$, $\sigma=10^{-1}$, $p=1,250$. Here, we consider the third-order moment tensor, $d=3$.
These are the results of the computations shown in the middle column of \cref{fig:gmm3}. For $\hat r < r$, it is of course impossible to recover all the means. In both cases, it is close to detecting the green and blue means.  For $\hat r = r$, the recovery is very good, despite the high overlap in the clusters. Most interestingly, overestimating the rank does not detract substantially from detecting the true means. 

\section{Extending to  Massive or Online Observations}
\label{sec:extend-probl-with}

In the case that we have many, many observations so that $p$ is
exceptionally large, even the implicit method may be too expensive.
Alternatively, we may have a situation where the data is online and needs
to be processed as it arrives.
In either case, an option is to use only a subset of observations
for each function or gradient evaluation.
For instance, we can show the following lemma, whose proof is left to the reader.
\begin{lemma}\label{lem:stoc}
  Let $\X = \frac{1}{p} \psum \Vlod \in \Kdnp$ where $\V \in \Real^{n \times p}$.
  Let $s \ll p$ and $\Vs \in \Real^{n \times s}$ comprise $s$ random columns (with replacement) from $\V$.
  Define the random stochastic Kruskal tensor $\Xs = \frac{1}{s} \sum_{\ell = 1}^s \Vc[\tilde]{v}_{\ell}^{\otimes d} \in \Kdns$.
  Then for a given vector $\Vc{a} \in \Real^n$, we have
  \begin{displaymath}
    \mathbb{E} [ \Xs\Vc{a}^{d-1} ] = \X\Vc{a}^{d-1}.
  \end{displaymath}
\end{lemma}
This enables us to compute stochastic functions and gradients.
If we define $\Mx[\tilde]{Y}$ such that its $j$th column is given by
$\Vc[\tilde]{y}_j=\Xs\Vc{a}^{d-1}$, then \cref{lem:stoc}
says $\mathbb{E}[\Mx[\tilde]{Y}] = \Y$. By linearity of expectation,
we can use this for unbiased stochastic estimators of the function and
gradients. Moreover, we can use \cref{alg:implicit} directly, only passing the stochastic $\Vs$ and weight vector $\Vc[\tilde]{\nu} =
\begin{bmatrix}
  1/s & 1/s & \cdots & 1/s
\end{bmatrix}^\intercal$ rather than the original values.

We provide a numerical example to demonstrate the potential of a stochastic
approach.
We revisit the Gaussian mixture model problem in \cref{sec:numer-estim-rank} and
consider an example with the following parameters: $n=500$, $\sigma=0.1$, $r=\hat r = 10$, and $d=3$. 
The major difference in the setup is that we assume we have $p=500,000$ samples. 

We compare the standard optimization used in \cref{sec:Results} to our own implementation of Adam \cite{KiBa15},
a variant of stochastic gradient descent. We employ a few standard modifications as follows.
We add the ability to compute an \emph{estimate} of the function value after each epoch.
We base our function value estimate on 1,000 randomly sampled observations --- using the same set for every experiment so that the function estimates are comparable.
We set the initial learning rate to $0.01$ and the epoch size to 100.
If the (estimated) function value does not decrease after an epoch,
the learning rate reduces to $0.001$
and the method resets to the iterate from the prior epoch before continuing.
If it happens again, the algorithm terminates as converged, reverting to the iterate from the prior epoch as the final solution.

\begin{table}
  \centering\footnotesize
  \begin{tabular}{{|r|S|S|S|}}
    \hline
\multicolumn{1}{|c|}{Method} & {Best $f$ (shifted)} &  {Sim. Score} & {Total Time (s)}\\ \hline 
    standard & -0.2471 & 0.9998 & 2166.70 \\ \hline 
  Adam, s=10 & -0.2209 & 0.9225 & 8.03 \\ \hline 
 Adam, s=100 & -0.2427 & 0.9929 & 10.48 \\ \hline 
Adam, s=1000 & -0.2464 & 0.9990 & 41.00 \\ \hline 
  \end{tabular}
  \caption{Comparison of standard and stochastic optimization (Adam) and varying numbers of samples ($s$)
    based on a rank $\hat r=10$ approximation for a moment tensor of order $d=3$.
    Data are generated by a spherical Gaussian mixture model with $r=10$ components and
    $p=500,000$ observations, evenly divided among the Gaussians.
    The observations are of size $n=500$, and the noise is $\sigma=0.1$.
    Each method is run ten times, reporting the best function value,
    corresponding similarity score, and total time for all ten runs. }
  \label{tab:stoch}
\end{table}

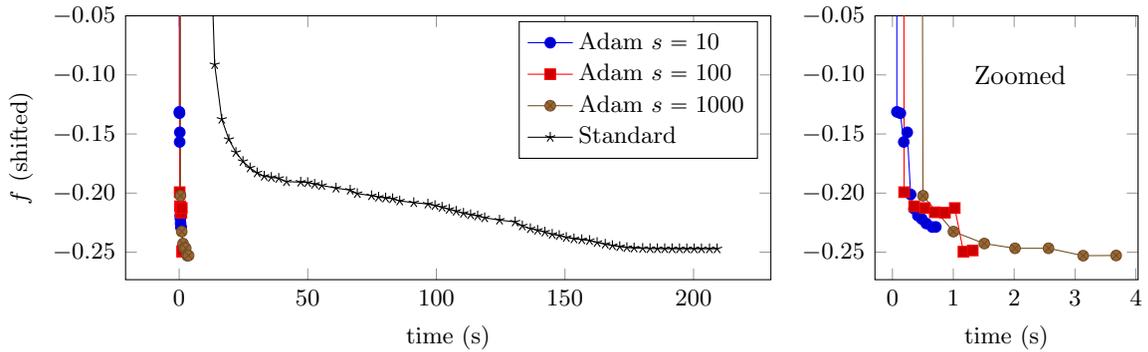
\begin{figure}
  \centering
  \pgfplotsset{
    height=2in,
    ymax = -0.05,
    y tick label style={
      /pgf/number format/.cd,
      fixed,
      fixed zerofill,
      precision=2,
      /tikz/.cd
    },      
    xlabel = time (s),
    tick label style={font=\footnotesize},
    label style={font=\footnotesize},
  }
  \begin{tikzpicture}
    \begin{axis}[
      width=4 in, 
      ylabel = {$f$ (shifted)},
      legend style={cells={anchor=west}, font=\footnotesize},
      ]
      \addplot table[x=time,y=fval] {stoch-trace-2.dat};
      \addplot table[x=time,y=fval] {stoch-trace-3.dat};
      \addplot table[x=time,y=fval] {stoch-trace-4.dat};
      \addplot table[x=time,y=fval] {stoch-trace-1.dat};
      \legend{Adam $s=10$, Adam $s=100$, Adam $s=1000$, Standard}
    \end{axis}
  \end{tikzpicture}
  \hfil
  \begin{tikzpicture}
    \begin{axis}[
      width=2 in, 
      ]
      \addplot table[x=time,y=fval] {stoch-trace-2.dat};
      \addplot table[x=time,y=fval] {stoch-trace-3.dat};
      \addplot table[x=time,y=fval] {stoch-trace-4.dat};
      \node[] at (axis cs:2.1,-0.1) {\small Zoomed};
    \end{axis}
  \end{tikzpicture}
  \caption{Comparison of the runs corresponding to the best solutions reported in \cref{tab:stoch}.
    For Adam, each marker indicates one epoch (100 iterations), and 
    the function values are estimates based on 1,000 samples.
    For the standard method, each marker indicates one iteration and the function values are exact.
    A zoomed-in plot is shown on the right to see the differences in the stochastic methods.
  }
  \label{fig:stoch-traces}
\end{figure}

In our experiments, we calculate only a \emph{shifted} version of $f$ from \cref{eq:optFunc}
to avoid calculating $\|\X\|^2$ (a constant term) since its complexity depends on $p^2$.
We run each method 10 times (using the same 10 initial guesses for each method), and we save the method that
achieves the lowest final (shifted) function value.
\Cref{tab:stoch} reports the function values and corresponding similarity scores as compared to the true solution.
The stochastic methods find solutions that are as good as the standard method at a tiny fraction of
the computational cost.
The progress at each iteration is shown in 
\cref{fig:stoch-traces}.

\section{Conclusions and future work}\label{sec:Conclusion}

The $d$th-order moment can be prohibitively expensive to form and decompose.
We developed a method to implicitly find a low-rank decomposition to the symmetric moment tensor that does not require forming the whole tensor. Our numerical results show that this new implicit method is returning the same decompositions as existing explicit methods, is considerably faster, and can solve problems that would have previously been intractable.

We have demonstrated the success of this method in the case of a spherical Gaussian mixture model for $n=500$.
In the case of massive or online data, we have provided preliminary results showing that stochastic calculation of the
gradient is also promising.
In contrast to \cite{AnGeHsKa14a}, we do not require whitening and give numerical evidence of the effectiveness of the method.
However, we note that the tensor approach as we have demonstrated can only find the direction of the means, not their norms.
We can incorporate information from the second-order moments to eliminate the ambiguity, e.g., as in \cite{HsKa13} for $d=3$.
Incorporating this approach is a topic for future research.

\section*{Acknowledgments}
We thank the anonymous referees for their enlightening comments and suggestions which have greatly improved the manuscript.
We also thank Daniel Hsu, Joe Kileel, and Jo{\~a}o M.~Pereira for helpful conversations.

\clearpage
\bibliographystyle{siamplain}


\end{document}